\documentclass[a4paper, 11pt]{amsart}
\usepackage[
  margin=1.9cm,
  includefoot,
  footskip=30pt,
]{geometry}
\usepackage{amsmath,amssymb,amscd}
\usepackage{amsfonts}
\usepackage[english]{babel}
\usepackage[leqno]{amsmath}
\usepackage{amssymb,amsthm}
\usepackage{mathrsfs} 
\usepackage{amscd}
\usepackage{enumerate}
\usepackage{epsfig}
\usepackage{relsize}
\usepackage{layout}
\usepackage[usenames,dvipsnames]{xcolor}
\usepackage[backref=page]{hyperref}
\usepackage{tikz}
\hypersetup{
 colorlinks,
 citecolor=Green,
 linkcolor=Red,
 urlcolor=Blue}
\usepackage[matrix,arrow,tips,curve]{xy}
\input{xy}
\xyoption{all}
\definecolor{darkgreen}{rgb}{0.0, 0.7, 0.0}
\definecolor{purple}{rgb}{0.5, 0.0, 0.5}
\definecolor{red}{rgb}{0.8, 0.2, 0.0}

\newtheorem{thm}{Theorem}[section]
\newtheorem{bthm}{Theorem}

\newtheorem{lemma}[thm]{Lemma}

\newtheorem{prop}[thm]{Proposition}

\newtheorem{claim}[thm]{Claim}

\numberwithin{equation}{section}
\theoremstyle{definition}
\newtheorem{defi}[thm]{Definition}

\theoremstyle{remark}
\newtheorem{remark}[thm]{Remark}

\newcommand{\Q}{\mathbb{Q}}

\def \Im{{\rm Im}}
\DeclareMathOperator{\Supp}{Supp}

\def \PP{\mathbb{P}}

\def \ZZ{\mathbb{Z}}

\def \P{\mathcal P}

\def \F{\mathcal F}

\def \L{\mathcal L}
\def \E{\mathcal E}
\def \G{\mathcal G}
\def \H{\mathcal H}
\def \U{\mathcal U}
\def\O{\mathcal O}

\def\M0{\mathcal M^0}

\DeclareMathOperator{\Proj}{{Proj}}
\DeclareMathOperator{\Sym}{{Sym}}

\def\B{\mathbf{B}}

\begin{document}

\title[On the classification of non-big Ulrich vector bundles on surfaces and threefolds]{On the classification of non-big Ulrich vector bundles on surfaces and threefolds}

\author[A.F. Lopez and R. Mu\~{n}oz]{Angelo Felice Lopez* and Roberto Mu\~{n}oz**}

\address{\hskip -.43cm Dipartimento di Matematica e Fisica, Universit\`a di Roma
Tre, Largo San Leonardo Murialdo 1, 00146, Roma, Italy. e-mail {\tt lopez@mat.uniroma3.it}}

\address{\hskip -.43cm Roberto Mu\~{n}oz, Departamento de Matem\'atica Aplicada, ESCET, Universidad Rey Juan Carlos, 28933 M\'ostoles (Madrid), Spain. e-mail {\tt roberto.munoz@urjc.es}} 

\thanks{* Research partially supported by  PRIN ``Advances in Moduli Theory and Birational Classification'' and GNSAGA-INdAM}

\thanks{** Research partially supported by Proyecto MTM2015-65968-P}


\begin{abstract} 
We classify Ulrich vector bundles that are not big on smooth complex surfaces and threefolds. 
\end{abstract}

\maketitle

\section{Introduction}

Let $X \subseteq \PP^N$ be a smooth complex projective variety. It is a well-known philosophy that the geometry of $X$ is reflected in the richness of its subvarieties. For this purpose a classical method is the use of vector bundles on $X$, the most useful being usually ample or at least big and globally generated vector bundles. The latter in turn can be produced by checking some cohomological vanishings via Castelnuovo-Mulmord regularity.

Strenghtening a little bit the condition one gets Ulrich vector bundles, that is vector bundles $\E$ such that $H^i(\E(-p))=0$ for all $i \ge 0$ and $1 \le p \le \dim X$.

The study of Ulrich vector bundles started in the 80's in commutative algebra \cite{u} and evolved recently in algebraic geometry through several pioneering works such as \cite{es, b1} (see also \cite{b2} and references therein). Nowadays many articles have appeared in the subject, focusing mainly in producing several examples - the basic existence question is still open - and, in some cases, even classifying them.

The first author started in \cite{lo} the investigation of positivity of Ulrich vector bundles. While in that paper the emphasis was on the positivity of the first Chern class, it appeared already clear that unless the variety or the vector bundle is somehow special, an Ulrich vector bundle should be positive enough. As a source of intuition, it was proved in \cite[Thm.~1]{lo} that an Ulrich vector bundle is big unless $X$ is covered by lines. Moreover other results (see \cite[Thm.~2, 3 and 4]{lo}) pointed in the direction of a possible classification of non-big Ulrich vector bundles.

It is the goal of the present research to give a classification of non-big Ulrich vector bundles on surfaces and threefolds.

In the case of surfaces, non-big Ulrich vector bundles turn out to be the same as the ones with $c_1(\E)^2=0$, as one can see from the ensuing

\begin{bthm}
\label{main1}

\hskip 3cm

Let $S  \subseteq \PP^N$ be a smooth irreducible complex surface. Let $\E$ be a rank $r$ Ulrich vector bundle on $S$. 

Then $\E$ is not big if and only if $(S,\O_S(1), \E)$ is one of the following:
\begin{itemize}
\item [(i)] $(\PP^2, \O_{\PP^2}(1), \O_{\PP^2}^{\oplus r})$;
\item [(ii)] $(\PP(\F), \O_{\PP(\F)}(1), \pi^*(\G(\det \F)))$, where $\F$ is a rank $2$ very ample vector bundle over a smooth curve $B$ and $\G$ is a rank $r$ vector bundle on $B$ such that $H^q(\G)=0$ for $q \ge 0$.
\end{itemize} 
\end{bthm}

On the other hand, on threefolds, as expected, the situation is more variegated, as there are examples with $c_1(\E)^3>0$. A beautiful classification result is given, in \cite{ahmpl}, on a rational normal scroll. In particular our classification on linear scrolls over a curve (see Theorem \ref{bei}) drew much inspiration from the above mentioned paper. 

Nevertheless we were able to classify them, as follows.

\begin{bthm}
\label{main2}

\hskip 3cm

Let $X \subseteq \PP^N$ be a smooth irreducible complex threefold. Let $\E$ be a rank $r$ vector bundle on $X$. 

Then $\E$ is Ulrich not big if and only if $(X, \O_X(1), \E)$ is one of the following. 

If $c_1(\E) = 0$:
\begin{itemize}
\item [(i)] $(\PP^3, \O_{\PP^3}(1), \O_{\PP^3}^{\oplus r})$; 
\end{itemize}

If $c_1(\E)^2=0, c_1(\E) \ne 0$:
\begin{itemize}
\item [(ii)] $(\PP(\F), \O_{\PP(\F)}(1), \pi^*(\G(\det \F)))$, where $\F$ is a rank $3$ very ample vector bundle over a smooth curve $B$, with $\G$ a rank $r$ vector bundle on $B$ such that $H^q(\G)=0$ for $q \ge 0$.
\end{itemize}
If $c_1(\E)^3 = 0, c_1(\E)^2 \ne 0$:
\begin{itemize}
\item [(iii)] $(\PP(\F), \O_{\PP(\F)}(1), \pi^*(\G(\det \F)))$, where $\F$ is a rank $2$ very ample vector bundle over a smooth surface $B$, with $\G$ a rank $r$ vector bundle on $B$ such that $H^q(\G) = H^q(\G \otimes \F^*)=0$ for $q \ge 0$.
\end{itemize}
If $c_1(\E)^3 > 0$:
\begin{itemize}
\item [(iv)] $(Q, \O_{Q}(1), \mathcal S)$, where $Q \subset \PP^4$ is a smooth quadric and $\mathcal S$ is the spinor bundle; 
\item[(v)] $(\PP(\F), \O_{\PP(\F)}(1), \E)$, where $\F$ is a rank $3$ very ample vector bundle over a smooth curve $B$, and $\E$ is an extension
$$0 \to \Omega_{X/B}(2H+\pi^*M) \to \E \to \pi^*(\G(\det \F)) \to 0$$
and $M$ is a line bundle, $\G$ is a rank $r-2$ vector bundle on $B$ such that $H^q(M)=H^q(\G)=0$ for every $q \ge0$.
\end{itemize}
\end{bthm}
As a matter of fact it is not difficult to see that all cases in the above theorem actually occur. See Remark \ref{cisono} for (iii) and Lemma \ref{exa} for (v).

Throughout the whole paper we work over the complex numbers.

We would like to thank G. Casnati and J.C. Sierra for many useful conversations.

\section{Notation and general facts on vector bundles}

\begin{defi}
Let $X$ be a smooth irreducible variety of dimension $n \ge 1$ and let $\E$ be a globally generated rank $r$ vector bundle on $X$. We let
$$\Phi_{\E} : X \to {\mathbb G}(r-1,\PP H^0(\E))$$
be the morphism to the Grassmannian. We set
$$\nu(\E) =  \nu(\O_{\PP(\E)}(1)) = \max\{k \ge 0 : \O_{\PP(\E)}(1)^k \ne 0\}$$
to be the numerical dimension of $\E$. We say that $\E$ is big if $\O_{\PP(\E)}(1)$ is big, that is if $\nu(\E) = r+n-1$.
\end{defi}

As we are interested in studying (non) big vector bundles, we will now collect some useful results.

\begin{remark}
\label{seg}
Let $X$ be a smooth projective variety of dimension $n$ and let $\E$ be a nef vector bundle. Then $s_n(\E^*) \ge 0$ and $\E$ is big if and only if $s_n(\E^*)>0$.
Moreover if $c_1(\E)^n=0$ then $\E$ is not big.
\end{remark}
\begin{proof} 
Let $\xi$ be the tautological line bundle on $\PP(\E) = \Proj(\Sym(\E))$ and let $\pi: \PP(\E) \to X$ be the projection. Then (see \cite[Def.~10.1]{eh}, where they use $\Proj(\Sym(\E^*))$), we have, setting $[P]$ for the class of a point $P \in \PP(\E)$ in the Chow group $A^{n+r-1}(\PP(\E)) \cong \ZZ$,
$$s_n(\E^*) = \pi_*([\xi^{n+r-1}]) = \pi_*(\xi^{n+r-1}[P]) = \xi^{n+r-1} \pi_*([P]) = \xi^{n+r-1}[\pi(P)]$$
that is, identifying with $\ZZ$ and observing that $\xi$ is nef, 
$$s_n(\E^*) = \xi^{n+r-1} \ge 0.$$
Moreover $\xi$ is big if and only if $\xi^{n+r-1} > 0$, that is if and only if $s_n(\E^*)>0$.

Now if $c_1(\E)^n=0$, it follows by \cite[Cor.~2.7]{dps} that $s_n(\E^*)=0$, therefore $\E$ is not big.
\end{proof} 

\begin{defi}
Let $X$ be a smooth irreducible variety and let $L$ be a line bundle on $X$. The augmented base locus of $L$ is
$$\B_+(L) = \bigcap\limits_{L=A+E} \Supp(E)$$
where the intersection is taken over all decompositions $L=A+E$, where $A$ is an ample $\Q$-divisor and $E$ is an effective $\Q$-divisor on $X$.
\end{defi}

The following result will be useful to prove non-bigness.

\begin{lemma}
\label{nonbig}
Let $X$ be a smooth irreducible variety and let $\E$ be a vector bundle on $X$. Let $\{Y_t\}_{t \in T}$ be a covering family of subvarieties of $X$. If $\E_{|Y_t}$ is not big for all $t \in T$, then $\E$ is not big.
\end{lemma}
\begin{proof}
Assume that $\E$ is big, so that $\B_+(\O_{\PP(\E)}(1)) \subsetneq \PP(\E)$. Let $\pi : \PP(\E) \to X$ be the projection and choose $z \in \PP(\E) \setminus \B_+(\O_{\PP(\E)}(1))$. Let $t \in T$ be such that $\pi(z) \in Y_t$. Then $z \in \pi^{-1}(Y_t) = \PP(\E_{|Y_t})$ and therefore $\PP(\E_{|Y_t})  \not\subseteq \B_+(\O_{\PP(\E)}(1))$. Hence $\O_{\PP(\E_{|Y_t})}(1) = \O_{\PP(\E)}(1)_{| \PP(\E_{|Y_t})}$ is big, that is $\E_{|Y_t}$ is big, a contradiction.
\end{proof} 

The following is the definition of Ulrich vector bundle.

\begin{defi}
Let $X \subseteq \PP^N$ be a smooth irreducible variety of dimension $n \ge 1$ and let $\E$ be a vector bundle on $X$. We say that $\E$ is an Ulrich vector bundle if $H^i(\E(-p))=0$ for all $i \ge 0$ and $1 \le p \le n$. We will sometimes say that $\E$ is an Ulrich vector bundle for $(X,H)$, where $H$ is an hyperplane section of $X \subseteq \PP^N$.
\end{defi}

Now a couple of remarks on Ulrich vector bundles. The first one will be often used without mentioning.

\begin{remark}
\label{ulr}
Let $X \subseteq \PP^N$ be a smooth irreducible variety of degree $d$ and let $\E$ be a rank $r$ Ulrich vector bundle on $X$. Then $\E$ is $0$-regular, hence it is globally generated by \cite[Thm.~1.8.5]{laz1}. Moreover $\E$ is ACM and $h^0(\E)=rd$ by \cite[(3.1)]{b2}.
\end{remark}
 
\begin{remark}
\label{spi}
Let $Q \subset \PP^4$ be a smooth quadric and let $\E$ be a rank $r$ Ulrich vector bundle on $Q$. Then $r$ is even and $\E \cong \mathcal S^{\oplus (\frac{r}{2})}$, where $\mathcal S$ is the spinor bundle. In particular $c_1(\E)^2=\frac{r^2}{2}L \ne 0$, where $L$ is a line in $Q$ and $s_3(\E^*)= \frac{r(r-2)(r+2)}{24}$. Hence $\E$ is big if and only if $r \ge 4$.
\end{remark}
\begin{proof} 
As is well-known (see \cite[Rmk.~2.5(4)]{bgs}, \cite[Rmk.~2.6]{b2}, \cite[Exa.~3.2]{ahmpl}), $\mathcal S$ has rank $2$ and is the only indecomposable Ulrich vector bundle on $Q \subset \PP^4$. Since any direct summand of $\E$ is an Ulrich vector bundle, it follows that $\E \cong \mathcal S^{\oplus (\frac{r}{2})}$. Now \cite[Rmk.~2.9]{ot} gives that $c_1(\mathcal S)=H$ and $c_2(\mathcal S)=L$, where $H$ is a hyperplane section of $Q$. It follows that $c_1(\E)^2=\frac{r^2}{2}L \ne 0$ and $s_3(\E^*)= \frac{r(r-2)(r+2)}{24}$. Hence Remark \ref{seg} gives that $\E$ is big if and only if $r \ge 4$.
\end{proof} 

\section{Non-big Ulrich vector bundles on surfaces}

In this section we classify non-big Ulrich vector bundles on surfaces.

\begin{proof}[Proof of Theorem \ref{main1}]

\hskip 3cm

If $(S, \O_S(1), \E)$ is as in (i) or (ii), it follows by  \cite[Prop.~2.1]{es} (or \cite[Thm.~2.3]{b2}) in case (i) and \cite[Lemma 4.1]{lo} in case (ii), that $\E$ is a rank $r$ Ulrich vector bundle on $S \subseteq \PP^N$ with $c_1(\E)^2=0$. Also $\E$ is not big by Remark \ref{seg}.

Vice versa let $\E$ be a rank $r$ Ulrich vector bundle on $S \subseteq \PP^N$ such that $\E$ is not big. 

If $(S, \O_S(1)) = (\PP^2,\O_{\PP^2}(1))$ it follows by \cite[Prop.~2.1]{es} (or \cite[Thm.~2.3]{b2}) that $\E = \O_{\PP^2}^{\oplus r}$, so that we are in case (i).

Assume from now on that $(S, \O_S(1)) \ne (\PP^2,\O_{\PP^2}(1))$.

We claim that $c_1(\E)^2=0$. 

Assume to the contrary that $c_1(\E)^2 \ne 0$. As $\E$ is globally generated, it follows that $c_1(\E)$ is nef and big, whence, in particular, $r \ge 2$. Let $d$ be the degree of $S$, so that $d \ge 2$. We have $h^0(\E) = rd \ge r+2$. Also observe that $H^1(-\det \E)=0$ by Kawamata-Viehweg's vanishing theorem. Now \cite[Thm.~3.1]{bf} implies that $\E$ is big, a contradiction.

This proves that $c_1(\E)^2=0$ and then \cite[Thm.~2]{lo} gives that $(S, \O_S(1), \E)$ is as in (ii).
\end{proof} 

\section{Non-big Ulrich vector bundles on $\PP^2$-bundles over a curve}

One basic source of non-big Ulrich vector bundles is on $\PP^{n-1}$-bundles over a curve. While it is easy, via pull-back, to find examples with $c_1(\E)^n=0$, the case $c_1(\E)^n>0$ is more delicate. Nevertheless, we will see below that, at least on $\PP^2$-bundles over a curve, the latter can be characterized as suitable extensions.

\begin{lemma}
\label{exa}
Let $B$ be a smooth irreducible curve and let $\F$ be a very ample rank $n \ge 2$ vector bundle on $B$. Let $X = \PP(\F)$ with tautological line bundle $\xi$ and projection $\pi: X \to B$. Let $M$ be a line bundle on $B$ such that $H^i(M)=0$ for every $i \ge0$. Set $H=\xi$. Then $\E:= \Omega_{X/B}(2H+\pi^*M)$ is an Ulrich vector bundle for $(X, H)$, $\E$ it is not big and, if $n \ge 3$, we also have that $c_1(\E)^n>0$. 
\end{lemma}
\begin{proof} 
Let $p$ be an integer such that $1 \le p \le n$. Then the twisted relative Euler sequence is
\begin{equation}
\label{eul}
0 \to \E(-pH) \to \pi^*(\F(M))((1-p)\xi) \to (2-p)\xi+\pi^*M \to 0.
\end{equation}
By \cite[Ex.~III.8.4]{ha} we see that
$$R^j \pi_*(\pi^*(\F(M))((1-p)\xi)) = 0 \ \mbox{if} \ j > 0 \ \hbox{or} \ j=0 \ \hbox{and} \ 2 \le p \le n$$
and
$$R^j \pi_*((2-p)\xi+\pi^*M) =  \begin{cases} 0 & {\rm if} \ j > 0 \ \hbox{or} \ j=0 \ \hbox{and} \ 3 \le p \le n \\ 
M & {\rm if} \  j=0 \ \hbox{and} \ p=2 \end{cases}.$$
Now the Leray spectral sequence gives that 
\begin{equation}
\label{leray1}
H^i(X, \pi^*(\F(M))((1-p)\xi)) = 0 \ \hbox {if} \ i \ge 2 \ \hbox{or} \ 0 \le i \le 1 \ \hbox{and} \ 2 \le p \le n
\end{equation}
and similarly
\begin{equation}
\label{leray2}
H^i(X, (2-p)\xi+\pi^*M) = \begin{cases} 0 & {\rm if} \ i \ge 2 \ \hbox{or} \ 0 \le i \le 1 \ \hbox{and} \ 3 \le p \le n \\ H^i(B, M)=0 & {\rm if} \  0 \le i \le 1 \ \hbox{and} \ p=2 \end{cases}.
\end{equation}
From the cohomology sequence of \eqref{eul}, using $\eqref{leray1}$ and $\eqref{leray2}$, we get that
$$H^i(\E(-pH))=0 \ \hbox{for} \ i \ge 0 \ \hbox{and} \ 2 \le p \le n.$$
Also, for $p=1$, we have that
$$\E(-H) = \Omega_{X/B}(\xi) \otimes \pi^*M$$
and it is a well-known fact (see \cite[Lemma 7.3.11(i) and (iii)]{laz2}) that
$$R^j \pi_*(\Omega_{X/B}(\xi))=0 \ \hbox{for} \ j \ge 0$$ 
whence the Leray spectral sequence gives that $H^i(X, \E(-H))=0$ for $i \ge 0$.
Thus $\E$ is Ulrich for $(X,H)$. 

Observe now that, on any fiber $F \cong \PP^{n-1}$ of $\pi$, we have that
$\E_{|F} \cong \Omega_{\PP^{n-1}}(2)$ is not big. In fact $\Omega_{\PP^{n-1}}(2)$ is $0$-regular, hence globally generated. If it were big, Griffiths' vanishing theorem \cite[Ex.~7.3.3]{laz2} would give the contradiction
$$0 = h^1(\omega_{\PP^{n-1}} \otimes \Omega_{\PP^{n-1}}(2) \otimes \det(\Omega_{\PP^{n-1}}(2))=h^1(\Omega_{\PP^{n-1}})=1.$$
Therefore Lemma \ref{nonbig} gives that $\E$ is not big. Finally set $g = g(B)$ and $d = \deg(\F)$. Then $\deg(M)=g-1$ and $H^n = d$. Now
$$c_1(\E) = c_1(\Omega_{X/B}) + 2(n-1)H + (n-1)\pi^*M \equiv (n-2)H + (d+(n-1)(g-1))F$$
we deduce that
$$c_1(\E)^n = (n-2)^{n-1}(n-1)(2d+n(g-1))>0$$
when $n \ge 3$.
\end{proof}

We now classify non-big Ulrich vector bundles in the case of $\PP^2$-bundles over a curve.

\begin{thm}
\label{bei}
Let $B$ be a smooth irreducible curve and let $\F$ be a very ample rank $3$ vector bundle on $B$. Let $X = \PP(\F)$ with tautological line bundle $\xi$ and projection $\pi: X \to B$. Set $H=\xi$. 

Let $\E$ be a rank $r$ vector bundle on $X$.

Then $\E$ is Ulrich for $(X,H)$ and not big if and only if $\E$ is as follows:

\begin{itemize}
\item[(i)] $\pi^*(\G(\det \F))$, with  $\G$ a rank $r$ vector bundle on $B$ such that $H^q(\G)=0$ for $q \ge 0$, when $c_1(\E)^2=0$, or 
\item[(ii)] $B = \PP^1, X = \PP^1 \times \PP^2, H = \O_{\PP^1}(1) \boxtimes \O_{\PP^2}(1)$ and $\E = \pi_2^*(\O_{\PP^2}(1))^{\oplus r}$, when $c_1(\E)^2 \ne 0, c_1(\E)^3 = 0$, or
\item[(iii)] an extension
\begin{equation}
\label{ext}
0 \to \Omega_{X/B}(2H+\pi^*M) \to \E \to \pi^*(\G(\det \F)) \to 0
\end{equation}
where $M$ is a line bundle and $\G$ is a rank $r-2$ vector bundle on $B$ such that $H^q(M)=H^q(\G)=0$ for every $q \ge0$, when $c_1(\E)^3 > 0$.
\end{itemize}
\end{thm}
\begin{proof} 
If $\E$ is as in (i) or (ii), then $\E$ is Ulrich for $(X,H)$ by \cite[Lemma 4.1]{lo}. Also $\E$ is not big by Remark \ref{seg}.

If $\E$ is as in (iii), then $\E$ is Ulrich for $(X,H)$ because so are $\Omega_{X/B}(2H+\pi^*M)$ by Lemma \ref{exa} and $\pi^*(\G(\det \F))$ by \cite[Thm.~3]{lo}. Moreover
$$c_1(\E) = c_1(\Omega_{X/B}(2H+\pi^*M)) + c_1(\pi^*(\G(\det \F))$$
and we know that $c_1(\pi^*(\G(\det \F))$ and $c_1(\Omega_{X/B}(2H+\pi^*M))$ are nef and $c_1(\Omega_{X/B}(2H+\pi^*M))^3>0$ by Lemma \ref{exa}. Therefore also $c_1(\E)^3>0$. Let $F \cong \PP^2$ be a fiber of $\pi$. Restricting  \eqref{ext} to $F$ we get an exact sequence
$$0 \to \Omega_{\PP^2}(2) \to \E_{|F} \to \O_{\PP^2}^{\oplus r-2} \to 0$$
which splits since $H^1(\Omega_{\PP^2}(2))=0$ and therefore 
$$\E_{|F} \cong \Omega_{\PP^2}(2) \oplus \O_{\PP^2}^{\oplus r-2}$$
is easily seen not to be big. But then Lemma \ref{nonbig} implies that $\E$ is not big.

Viceversa assume now that $\E$ is Ulrich for $(X,H)$ and that $\E$ is not big. 

If $c_1(\E)^2=0$ it follows by \cite[Proof of Thm.~3]{lo} (or \cite[Cor.~1]{ls}) that $\E$ is as in (i). 

Suppose then that $c_1(\E)^2 \ne 0$.

Let $\tau_0, \ldots, \tau_{r-1} \in H^0(\E)$ be general sections, thus giving rise to a general morphism $\phi: \O_X^{\oplus r} \to \E$. Note that the expected codimension of the degeneracy locus $D_{r-2}(\phi)$ is $4$. Therefore it follows by \cite[Statement(folklore), \S 4.1]{ba} that $D_{r-2}(\phi) = \emptyset$ and that, at any point where $\phi$ drops rank, it has rank $r-1$. Thus we have an exact sequence
\begin{equation}
\label{elle}
0 \to \O_X^{\oplus r} \to \E \to \L \to 0
\end{equation}
where $\L$ is a line bundle on $W = D_{r-1}(\phi)$. Moreover, again \cite[Statement(folklore), \S 4.1]{ba} implies that $W$ is smooth. We now claim that $W$ is irreducible. In fact consider the map
$$\lambda_{\E} : \Lambda^r H^0(\E) \to H^0(\det \E).$$
Then $|\Im \lambda_{\E}|$ is base-point free and not composite with a pencil since  $c_1(\E)^2 \ne 0$.
Therefore any element of $|\Im \lambda_{\E}|$, hence also $W$, is connected. It follows that $W$ is a smooth irreducible surface. 

Set $d = H^3, g = g(B)$ and let $F$ be a fiber of $\pi$. Note that $d \ge 3$ for otherwise $X$ would be $\PP^3$ or a quadric, contradicting the fact that $\rho(X)=2$. There are integers $a$ and $b$ and a line bundle $N$ on $B$ such that
\begin{equation}
\label{vdoppio}
W \sim \det \E \sim a\xi + \pi^*N \equiv a \xi + b F.
\end{equation}
Hence $a \ge 0$ and $0 \ne W^2 = a^2 \xi^2 + 2ab \xi F$, so that $a \ge 1$.

\begin{claim}
\label{menodue}
$(W, H_{|W}) \ne (\PP^2, \O_{\PP^2}(1))$. In particular $F \not\subseteq W$.
\end{claim}
\begin{proof}
Assume that $(W, H_{|W}) \cong (\PP^2, \O_{\PP^2}(1))$. Let $R$ be a line in $W \cong \PP^2$, so that 
$$1 = H_{|W} \cdot R = H \cdot R = \xi \cdot R.$$
But then the adjunction formula and \eqref{vdoppio} give
$$-3 = R \cdot K_{W} = R \cdot (K_X+W) = R \cdot [(a-3)\xi + (2g-2+d+b)F]$$
that is 
$$a + b R \cdot F + (2g-2+d) R \cdot F = 0.$$
Since $d \ge 3, R \cdot F \ge 0$ and
$$0 \le W \cdot R =  \det \E \cdot R = a + bR \cdot F$$
we get that $R \cdot F=0$ and then $a = 0$, a contradiction.
\end{proof}

\begin{claim}
\label{menouno}
$\L$  is an Ulrich line bundle for $(W, H_{|W})$ and $\L^2 = 0$. 
\end{claim}
\begin{proof} 
Note that $H^i(\E(-pH)=H^i(\O_X(-pH))=0$ for $i \ge 0$ and $p=1, 2$. Then \eqref{elle} gives that 
$H^i(\L(-pH_{|W})=0$ for $i \ge 0$ and $p=1, 2$, that is $\L$ is an Ulrich line bundle for $(W, H_{|W})$.

Let 
$$V = \varphi_{\O_{\PP(\E)}(1)}(\PP(\E)) \subseteq \PP H^0(\O_{\PP(\E)}(1)).$$ 
Since $\E$ is not big it follows that $\dim V \le r+1$. On the other hand \eqref{elle} implies that there is an inclusion $W \hookrightarrow \PP(\E)$ such that ${\O_{\PP(\E)}(1)}_{|W} \cong \L$ and that 
$$\varphi_{\O_{\PP(\E)}(1)}(W) \subseteq V \cap H_1 \cap \ldots \cap H_r$$
for general hyperplanes $H_1, \ldots, H_r$. But then
$$\dim \varphi_{\L}(W) \le \dim V - r \le 1.$$
As $\L$ is globally generated it follows that $\L^2 = 0$. 
\end{proof}

Using Claim \ref{menouno}, \cite[Thm.~2]{lo} and Claim \ref{menodue} we get that 
\begin{equation}
\label{mio}
(W, H_{|W}, \L) \cong (\PP(\F''), \O_{\PP(\F'')}(1) \otimes p^*L, p^*(M'+L+\det \F''))
\end{equation}
is a $\PP^1$-bundle $p : W \cong \PP(\F'') \to C$, with $C$ a smooth curve, $\F''$ a rank two vector bundle on $C$, $M', L$ two line bundles on $C$ with $H^i(M'-L)=0$ for every $i \ge 0$. 

Set $\xi' = \O_{\PP(\F'')}(1)$ and let $f$ be a fiber of $p$.

\begin{claim}
\label{num}
If $f \cdot F=0$ then $a = 1$. If $f \cdot F \ge 1$, then $b = -a, d=3, g=0$ and $c_1(\E)^3 = 0$. 
\end{claim}
\begin{proof} 
By \eqref{mio} we have that $H_{|W} \equiv \xi'+uf$ for some $u \in \ZZ$, hence
$$1 = \xi' \cdot f = H_{|W} \cdot f = H \cdot f = \xi \cdot f.$$
By the adjunction formula and \eqref{vdoppio} we deduce that
$$-2 = f \cdot K_{W} = f \cdot (K_X+W) = f \cdot [(a-3)\xi + (2g-2+d+b)F] = a-3+(2g-2+d+b) f \cdot F$$
hence
\begin{equation}
\label{effe}
a + bf \cdot F + (2g-2+d) f \cdot F = 1.
\end{equation}
If $f \cdot F=0$ we have that  $a = 1$. If $f \cdot F \ge 1$, since
$$0 \le W \cdot f =  \det \E \cdot f = a + bf \cdot F$$
we get from \eqref{effe} that  $a + bf \cdot F=0, f \cdot F = 1$ and $2g-2+d=1$. The latter imply that $b=-a, d=3$ and $g=0$. This gives that $c_1(\E)^3 = 0$.
\end{proof}

We now conclude the case $f \cdot F \ge 1$. As it will turn out, it is the only case with $c_1(\E)^3 = 0$.
\begin{claim}
\label{cubozero}
If $f \cdot F \ge 1$, then $(X, H, \E)$ is as in (ii).
\end{claim}
\begin{proof} 
By Claim \ref{num} we know that $d=3$ and $g=0$, that is $\F \cong \O_{\PP^1}(1)^{\oplus 3}$ and then $X = \PP^1 \times \PP^2$ and $H = \O_{\PP^1}(1) \boxtimes \O_{\PP^2}(1)$. Moreover $\det \E \sim a(\xi-F) \sim \pi_2^*(\O_{\PP^2}(a))$ and \cite[Lemma  5.1]{lo} gives that there is a rank $r$ vector bundle $\H$ on $\PP^2$ such that $\E \cong  \pi_2^*\H$. Setting $\H' =\H(-1)$ we deduce by \cite[Lemma  4.1(ii)]{lo} that $\H'$ is a rank $r$ Ulrich vector bundle for $(\PP^2, \O_{\PP^2}(1))$, whence that $\H' \cong \O_{\PP^2}^{\oplus r}$ and therefore $\E \cong \pi_2^*(\O_{\PP^2}(1))^{\oplus r}$.
\end{proof}
We will assume, from now on, that $f \cdot F=0$. Recall that $a = 1$ by Claim \ref{num}. 

\begin{claim}
\label{secondo}
There is a rank two vector bundle $\F'$ on $B$ and a surjection $\F \twoheadrightarrow \F'$ such that
$$(W, H_{|W}, \L) \cong (\PP(\F'), \O_{\PP(\F')}(1), \pi_{|W}^*(M+\det \F'))$$ 
for some line bundle $M$ on $B$ such that $H^i(M)=0$ for every $i \ge 0$.
\end{claim}
\begin{proof} 
Since $f \cdot F=0$ it follows that any fiber $f$ of $p$ is contained in a fiber $F$ of $\pi$ and \cite[Lemma 1.15(b)]{de} allows to construct a morphism $\psi : C \to B$ such that $\pi_{|W} = \psi \circ p$. Also $H \cdot W \cdot F = 1$ hence $W \cap F$ is a line, for every $F$. Therefore, for every $b \in B$ we have that
$$W \cap \pi^{-1}(b) = \pi_{|W}^{-1}(b) = p^{-1}(\psi^{-1}(b))$$
is a line, hence $\psi$ must have degree $1$, that is $\psi$ is an isomorphism. Hence we can assume that $C=B$ and $p = \pi_{|W}$. Now $\xi_{|W} = H_{|W} = \xi'+\pi_{|W}^*L$ by \eqref{mio} and this gives the exact sequence 
\begin{equation}
\label{rest}
0 \to \O_X(\xi-W) \to \O_X(\xi) \to \O_W(\xi'+\pi_{|W}^*L) \to 0.
\end{equation}
Moreover we know that $W \sim \xi + \pi^*N$, for some line bundle $N$ on $B$ and then $R^1\pi_*(\O_X(\xi-W)) = 0$. Appying $\pi_*$ to \eqref{rest} we get the exact sequence 
$$0 \to -N \to \F \to \F''(L) \to 0$$ 
and setting $\F' = \F''(L)$ and $M=M'-L$ we get that $W \cong \PP(\F')$, $H_{|W} \cong \O_{\PP(\F')}(1)$ and $\L \cong \pi_{|W}^*(M+\det \F')$ with $H^i(M)=0$ for every $i \ge 0$.
\end{proof}

\begin{claim}
\label{terzo}
$\E_{|F} \cong \O_{\PP^2}^{\oplus r-2} \oplus \Omega_{\PP^2}(2)$.
\end{claim}
\begin{proof} 
Let $R$ be any line in $F \cong \PP^2$. It follows by Claim \ref{secondo} that $W \cdot R = 1$. Then $\E_{|R}$ is globally generated and $c_1(\E_{|R})=1$, that is $\E_{|R} \cong \O_{\PP^1}^{\oplus r-1} \oplus \O_{\PP^1}(1)$. Now \cite[Thm.~5.1]{e} implies that either $\E_{|F} \cong \O_{\PP^2}^{\oplus r-2} \oplus \Omega_{\PP^2}(2)$ or $\E_{|F} \cong \O_{\PP^2}^{\oplus r-1} \oplus \O_{\PP^2}(1)$.

On the other hand, as $F \not\subseteq W$ by Claim \ref{menodue}, we have that ${\tau_0}_{|F}, \ldots, {\tau_{r-1}}_{|F} \in H^0(\E_{|F})$ are linearly independent at the general point of $F$ and therefore we have an exact diagram, obtained from \eqref{elle},
$$\xymatrix{& 0 \ar[d] & 0 \ar[d] & 0 \ar[d] & \\ 0 \ar[r] & \O_X(-F)^{\oplus r} \ar[r] \ar[d] & \E(-F) \ar[r] \ar[d] & \L \otimes \O_X(-F)  \ar[r] \ar[d] & 0 \\ 0 \ar[r] & \O_X^{\oplus r} \ar[r]  \ar[d] & \E \ar[r]  \ar[d] & \L \ar[r]  \ar[d] & 0 \\ 0 \ar[r] & \O_F^{\oplus r} \ar[r] \ar[d] & \E_{|F} \ar[r] \ar[d] & \L_{|F} \ar[r] \ar[d] & 0 \\ & 0 & 0 & 0 & }.$$
Observe that Claim \ref{secondo} gives that $W \cap F$ is a line and that $\L_{|F} \cong \O_{\PP^1}$. Then the above diagram gives an exact sequence
$$0 \to \O_{\PP^2}^{\oplus r} \to \E_{|F} \to \O_{\PP^1} \to 0.$$
But the latter implies that $H^0(\E_{|F}(-1))=0$, therefore excluding the case $\O_{\PP^2}^{\oplus r-1} \oplus \O_{\PP^2}(1)$.
\end{proof}

\begin{claim}
\label{quarto}
Let $M$ be the line bundle in Claim \ref{secondo} and set $\mathcal M = \E(-\xi-\pi^*M)$. For any $b \in B$ set $F_b = \pi^{-1}(b)$. Then we have:
\begin{itemize}
\item[(i)] $H^2(\mathcal M(j\xi-F_b))=0$ for $j = 0, -2$; 
\item[(ii)] $H^1(\mathcal M(j\xi-F_b))=0$ for $j = 0, -1, -2$; 
\item[(iii)] $H^1(\mathcal M(-j\xi))=0$ for $j = 0, -2$;
\item[(iv)] $h^1(\mathcal M(-\xi)) = 1$;
\item[(v)] $H^0(\mathcal M) = 0$. 
\end{itemize}
\end{claim}
\begin{proof} 
We start with an observation that we will use throughout the proof. 

Let $D=a_1P_1+\ldots+a_sP_s$ be an effective divisor on $B$, where $P_1,\ldots,P_s \in B$ are distinct points and $a_k \ge 1$ for $1 \le k \le s$. Let $\H$ be a vector bundle on $X$. Given an $i \ge 0$, to prove that $H^i(\H_{|\pi^*D})=0$ it is enough to show that $H^i(\H_{|F})=0$ for any fiber $F$.

In fact $\pi^*D = a_1F_{P_1}+\ldots+a_sF_{P_s}$ and 
\begin{equation}
\label{mult}
H^i(\H_{|\pi^*D})=\bigoplus\limits_{k=0}^s H^i(\H_{|a_kF_{P_k}}).
\end{equation}
On the other hand, using for every $c \ge 1$, the exact sequence 
$$0 \to \O_F \to \O_{cF} \to \O_{(c-1)F} \to 0$$
it is clear that if $H^i(\H_{|F})=0$ then $H^i(\H_{|cF})=0$ by induction on $c$. Therefore also $H^i(\H_{|\pi^*D})=0$ by \eqref{mult}.

Since $\E$ is ACM, we have that $H^i(\E(j\xi))=0$ for $i=1, 2$ and for every $j \in \ZZ$.

To see (i) and (ii), note that when $g=0$ we have that $\mathcal M(j\xi-F_b) = \E((j-1)\xi)$, hence (i) and (ii) hold. When $g \ge 1$ we have that $\deg(M+b)=g$ so that there exists an effective divisor $D \sim M+b$ and an exact sequence
$$0 \to \mathcal M(j\xi-F_b) \to \E((j-1)\xi) \to \E((j-1)\xi)_{|\pi^*D} \to 0.$$
Therefore the above observation gives that, to prove (i) and (ii), it is enough to show that 
$$H^1(\E((j-1)\xi)_{|F})=0 \ \mbox{for} \ j = 0, -2 \ \mbox{and} \ H^0(\E((j-1)\xi)_{|F})=0  \ \mbox{for} \ j = 0, -1, -2.$$ 
But these vanishings follow easily from Claim \ref{terzo}. Hence (i) and (ii) are proved.

Now let $j = 0, -2$. Then $H^1(\mathcal M(j\xi-F_b))=0$ by (ii) and $H^1(\mathcal M(j\xi)_{|F_b}) = H^1(\E((j-1)\xi)_{|F_b})=0$ by Claim \ref{terzo}. Thus the exact sequence
$$0 \to \mathcal M(j\xi-F_b) \to  \mathcal M(j\xi) \to \mathcal M(j\xi)_{|F_b} \to 0$$
gives (iii). To see (iv) consider the exact sequence, obtained from \eqref{elle},
$$0 \to \O_X(-2\xi-\pi^*M)^{\oplus r} \to \mathcal M(-\xi) \to \L(-2\xi-\pi^*M) \to 0.$$
Since $H^i(\O_X(-2\xi-\pi^*M))=0$ for $i=1,2$ we get, using Claim \ref{secondo}, that
$$h^1(\mathcal M(-\xi)) = h^1(\L(-2\xi-\pi^*M))=h^1(\O_{\PP(\F')}(-2) \otimes \pi_{|W}^*(\det \F'))=$$
$$=h^0(R^1(\pi_{|W})_*(\O_{\PP(\F')}(-2)) \otimes \O_B(\det \F')) = h^0(\O_B)=1.$$
Thus (iv) is proved. Moreover the exact sequence, obtained from \eqref{elle},
$$0 \to \O_X(-\xi-\pi^*M)^{\oplus r} \to \mathcal M \to \L(-\xi-\pi^*M) \to 0$$
implies that $H^0(\mathcal M)=0$, that is (v), because $H^0(\O_X(-\xi-\pi^*M))=0$ and $H^0(\L(-\xi-\pi^*M))=H^0(\O_{\PP(\F')}(-1) \otimes \pi_{|W}^*(\det \F'))=0$.
\end{proof}
Let $\mathcal M$ be as in Claim \ref{quarto}. By \cite[Thm.~4.1]{ab} there is a second quadrant spectral sequence 
\begin{equation}
\label{ssp0}
\{E_1^{p,q}, -3 \le p \le 0 \le q \le 3\} \ \mbox{abutting to} \ \mathcal M \ \mbox{when} \ p+q=0
\end{equation}
with the following properties. There is an exact sequence
\begin{equation}
\label{ssp}
\cdots \to R^q {p_1}_*(h^*(-\Delta) \otimes p_2^*\mathcal M((p+1)\xi)) \otimes \Omega_{X/B}^{-p-1}(-(p+1)\xi) \to E_1^{p,q} \to H^q(\mathcal M(p\xi)) \otimes \Omega_{X/B}^{-p}(-p\xi) \to
\end{equation}
$$\to R^{q+1}{p_1}_*(h^*(-\Delta) \otimes p_2^*\mathcal M((p+1)\xi)) \otimes \Omega_{X/B}^{-p-1}(-(p+1)\xi) \to \cdots$$

where $p_1, p_2 : X \times X \to X$ are the projections, $\Delta$ is the diagonal in $B \times B$ and $h= \pi \times \pi$.
Moreover
\begin{equation}
\label{ssp2}
E_1^{0,q} \cong H^q(\mathcal M) \otimes \O_X
\end{equation}
and
\begin{equation}
\label{ssp3}
E_1^{-3,q} \cong R^q{p_1}_*(h^*(-\Delta) \otimes p_2^*\mathcal M(-2\xi)) \otimes \O_X(-\xi+\pi^*(\det \F)).
\end{equation}

\begin{claim}
\label{quinto}
With the above notation we have:
\begin{itemize}
\item[(i)] $E_1^{0,0} = E_1^{0,1} = E_1^{-2,1} = E_1^{-3,2} = 0$; 
\item[(ii)] $E_{\infty}^{-1,1} = E_1^{-1,1} \cong \Omega_{X/B}(\xi)$. 
\end{itemize}
\end{claim}
\begin{proof} 
By Claim \ref{quarto}(iii)-(v) we have that $H^0(\mathcal M)=H^1(\mathcal M)=H^1(\mathcal M(-2\xi))=0$. It follows by \eqref{ssp2} that $E_1^{0,0}=E_1^{0,1} = 0$. Moreover Claim \ref{quarto}(ii) gives that $H^1(\mathcal M(-\xi-F_b))=H^2(\mathcal M(-2\xi-F_b))=0$ and then \cite[Rmk.~4.2]{ab} implies that $R^1{p_1}_*(h^*(-\Delta) \otimes p_2^*\mathcal M(-\xi))=R^2{p_1}_*(h^*(-\Delta) \otimes p_2^*\mathcal M(-2\xi))=0$. Therefore $E_1^{-2,1} = 0$ by \eqref{ssp} and $E_1^{-3,2} = 0$ by \eqref{ssp3}. This proves (i). To see (ii) consider, for $s \ge 1$, the differentials
$$E_s^{-1-s,s} \to E_s^{-1,1} \to E_s^{-1+s,2-s}.$$
We have that $E_s^{-1+s,2-s}=0$ because either $s=1$ and this is (i), or $s \ge 2$ and then $-1+s \ge 1$. On the other hand we have that $E_s^{-1- s,s}=0$ because either $s=1, 2$ and this follows by (i), or $s \ge 3$ and then $-1- s \le -4$. Therefore 
$E_{\infty}^{-1,1} = E_1^{-1,1}$. Moreover again \cite[Rmk.~4.2]{ab} and Claim \ref{quarto}(i)-(ii) imply that $R^1{p_1}_*(h^*(-\Delta) \otimes p_2^*\mathcal M)=R^2{p_1}_*(h^*(-\Delta) \otimes p_2^*\mathcal M)=0$. It follows from \eqref{ssp} and Claim \ref{quarto}(iv) that $E_1^{-1,1} \cong \Omega_{X/B}(\xi)$. This is (ii).
\end{proof}
We now proceed to conclude the proof of the theorem.

Consider the filtration associated to the spectral sequence \eqref{ssp0}
$$\mathcal M = F^3 \supseteq F^2 \supseteq F^1 \supseteq F^0 \supseteq F^{-1}=0.$$
It follows by Claim \ref{quinto} that $F^0 = F^0/F^{-1} = E_{\infty}^{0,0} = 0$ and then $F^1 = F^1/F^0 = E_{\infty}^{-1,1} = E_1^{-1,1} \cong \Omega_{X/B}(\xi)$. But then we have an inclusion $\Omega_{X/B}(\xi) = F^1 \subseteq \mathcal M$ and therefore also an inclusion 
$$\Omega_{X/B}(2\xi+\pi^*M) \hookrightarrow \E.$$ 
Let $\mathcal Q$ be the quotient and consider the exact sequence
\begin{equation}
\label{q}
0 \to \Omega_{X/B}(2\xi+\pi^*M) \to \E \to \mathcal Q \to 0.
\end{equation}
%
Now Lemma  \ref{exa} gives that both $\Omega_{X/B}(2\xi+\pi^*M)$ and $\E$ are Ulrich vector bundles for $(X,H)$, hence so is $\mathcal Q$ by \cite[Prop. 2.14]{ckm}. Moreover, using Claim \ref{secondo}, we see that $c_1(\mathcal Q) = c_1(\E)-c_1(\Omega_{X/B}(2\xi+\pi^*M)) \equiv vF$ for some $v \in \ZZ$. But then $c_1(\mathcal Q)^2=0$ and it follows by \cite[Proof of Thm.~3]{lo} (or \cite[Cor.~1]{ls}) that $\mathcal Q \cong \pi^*(\G(\det \F))$ where $\G$ is a rank $r-2$ vector bundle on $B$ such that $H^i(\G)=0$ for every $i \ge0$. 

This gives case (iii). 
\end{proof}

\section{Non-big Ulrich bundles on threefolds}

In this section we will prove Theorem \ref{main2}. We start with a couple of consequences of \cite{ls}.

\begin{lemma}
\label{riga}

Let $X  \subseteq \PP^N$ be a smooth irreducible threefold. Suppose that $X$ carries a rank $r$ Ulrich vector bundle $\E$ such that $\E$ is not big and $c_1(\E)^3>0$. 

Then $(X,\O_X(1))$ is one of the following:
\begin{itemize}
\item [(i)] $(Q, \O_Q(1)$, where $Q \subset \PP^{n+1}$ is a quadric;
\item [(ii)] $(\PP(\F), \O_{\PP(\F)}(1))$, where $\F$ is a very ample rank $3$ vector bundle on a smooth curve.
\end{itemize}
\end{lemma}
\begin{proof}
Set $\nu = \nu(\E)$. By hypothesis we have that $\nu \le r+1$. Let $m$ be the dimension of the family of lines in $X$ passing through a general point. By \cite[Cor.~2]{ls} we get that $m \ge r-\nu+2 \ge 1$. Note that $(X,\O_X(1)) \ne (\PP^3, \O_{\PP^3}(1))$, for otherwise it follows by \cite[Prop.~2.1]{es} (or \cite[Thm.~2.3]{b2}) that $\E \cong \O_{\PP^3}^{\oplus r}$, contradicting the fact that $c_1(\E)^3>0$. Now \cite[Thm.~1.4]{lp} gives that $(X,\O_X(1))$ is as in (i) or (ii).
\end{proof}

Before giving the next result, let us adopt the following

\begin{defi}
Let $X$ be a smooth irreducible variety and let $H$ be a very ample line bundle on $X$. We say that $(X, H)$ is a linear $\PP^k$-bundle over a smooth variety $Y$ if there is a rank $k+1$ vector bundle $\F$ on $Y$ and a line bundle $L$ on $Y$ such that $(X, H) \cong (\PP(\F), \O_{\PP(\F)}(1) + \pi^*L)$, where $\pi : \PP(\F) \to Y$ is the projection.
\end{defi}

\begin{prop}
\label{surf}

Let $S$ be a smooth irreducible surface and let $\F$ be a very ample rank $2$ vector bundle on $S$. Let $X = \PP(\F)$ with tautological line bundle $\xi$ and projection $\pi: X \to S$. Set $H=\xi$. Assume that $(X, H)$ is not a linear $\PP^2$-bundle over a smooth curve.

Let $\E$ be a rank $r$ Ulrich vector bundle for $(X,H)$ such that $c_1(\E)^3=0$ and $c_1(\E)^2 \ne 0$.

Then there are a smooth irreducible surface $B$, a very ample rank $2$ vector bundle $\G$ on $B$, a rank $r$ vector bundle $\H$ on $B$ with $H^q(\H)=H^q(\H \otimes \G^*)=0$ for $q \ge 0$, so that $X \cong \PP(\G), H \cong \O_{\PP(\G)}(1)$ and $\E \cong p^*(\H(\det \G))$ where $p: \PP(\G) \to B$ is the projection.
\end{prop}

\begin{proof}
It follows by \cite[Thm.~1]{ls} that the general fiber of $\Phi_{\E} : X \to {\mathbb G}(r-1,\PP H^0(\E))$ is a line $L$ and all fibers are linear spaces in $\PP^N$. Moreover $\E_{|L}$ is trivial, hence $L \cdot \det \E = 0$.

Suppose first that all fibers of $\Phi_{\E}$ are $1$-dimensional.

Consider the Stein factorization 
$$\xymatrix{X \ar[dr]_{\Phi_{\E}} \ar[r]^p & B \ar[d]^g \\ & \Phi_{\E}(X)}$$
where $B$ is normal and $g$ is finite. Since the general fiber of $\Phi_{\E}$ is reduced and irreducible, it follows that $g$ is bijective, so that the fibers of $p$ coincide with the fibers of $\Phi_{\E}$. Now we can apply \cite[Prop.~3.2.1]{bs} and get that $B$ is a smooth surface and $p : X \cong \PP(\G) \to B$ is a linear $\PP^1$-bundle. We know that $\E = \Phi_{\E}^*\U = p^*(g^*\U)$, where $\U$ is the tautological bundle on $\mathbb G(r-1, \P H^0(\E))$. Setting $\H = g^*\U(-\det \G)$, it follows by \cite[Lemma 4.1]{lo} that we are done in this case. 

Suppose now that there is a fiber $F = \PP^2$ of $\Phi_{\E}$.

Let $f = f_x$ be a fiber of $\pi$ passing through $x \in X$. Note that $f$ is a line in $\PP^N$. 

We claim that $f \cdot F = 1$. In fact it cannot be that for every $x \in F$ we have that $f_x \subset F$, for otherwise two different fibers $f_x, f_{x'}$ would meet. Then there exists an $x \in F$ such that $f_x \not\subset F$ and therefore $f \cdot F = 1$.

Thus we deduce that $\pi_{|F} : F \to S$ is an isomorphism, that is $S = \PP^2$. 

Pick a line $Z$ in $\PP^2$ and let $R = \pi^*Z$. There are two integers $a$ and $b$ such that 
$$\det \E \sim a \xi + b R.$$
If $a=0$ we get that $\det \E = \pi^*(\O_{\PP^2}(b))$ and it follows by \cite[Lemmas 5.1 and 4.1]{lo} that we are done in this case.

Suppose from now on that $a \ne 0$. 

Note that $L \cdot R \ge 0$ and
\begin{equation}
\label{int}
0 = L \cdot \det \E = a + b L \cdot R
\end{equation}
hence $L \cdot R \ge 1$.
 
If $\F$ is indecomposable, using the dominating unsplit family of lines given by the fibers of $\Phi_{\E}$, we get by \cite[Thm.~6.5]{mos} that $X \cong \PP(T_{\PP^2})$ with tautological line bundle $\xi'$ and $H = \xi = \xi' + lR$ for some $l \in \ZZ$. Note that 
$$4 \le h^0(H) = h^0(T_{\PP^2}(l))$$
hence $l \ge 0$. If $l \ge 1$, as $\xi'$ is very ample, we get the contradiction
$$1 = H \cdot L = \xi' \cdot L + l L \cdot R \ge 2.$$
Therefore $l = 0$ and we conclude in this case by \cite[Prop.~6.2]{lo}.

It remains to consider the case $\F = \O_{\PP^2}(\alpha) \oplus \O_{\PP^2}(\beta)$ for some $\alpha \ge \beta \ge 1$. Note that 
$$\PP(\O_{\PP^2}(\alpha)) \sim \xi - \beta R$$
hence
$$0 \le L \cdot (\xi - \beta R) = 1 - \beta L \cdot R$$
and it follows that $\beta = L \cdot R = 1$ and \eqref{int} gives that $b=-a$. Now 
$$\xi^3 = \alpha^2+\alpha +1, \xi^2 \cdot R = \alpha +1, \xi \cdot R^2 = 1 \ \hbox{and} \ R^3=0$$
hence
$$0 = (\det \E)^3 = a^3(\alpha-1)^2$$
so that $\alpha=1$. But then $(X, H)$ is a linear $\PP^2$-bundle over a smooth curve. As this case is excluded, the proof is complete.
\end{proof}

\begin{remark}
If the two $\PP^1$-bundle structures on $X$ are different, it follows by \cite[Thm.~B]{sa} that either $S$ and $B$ are $\PP^1$-bundles over the same smooth curve $C$ and $X \cong S \times_C B$ or $S=B=\PP^2$ and $X \cong \PP(T_{\PP^2})$.
\end{remark}

\begin{remark}
\label{cisono}
Ulrich vector bundles as in Proposition \ref{surf} can occur. For example (see \cite[Lemma 4.1]{lo}) one can take a very ample line bundle $L$ on $B$, $\G = L^{\oplus 2}$ and $\H$ a rank $r$ vector bundle on $B$ such that $\H(L)$ is Ulrich for $(B, L)$. Another possibility is when $B = \PP^2$ and $\G= T_{\PP^2}$ (see \cite[Prop.~6.2]{lo}).
\end{remark}

We are now ready for our classification theorem.

\begin{proof}[Proof of Theorem \ref{main2}]

\hskip 3cm

If $(X, \O_X(1), \E)$ is as in (i)-(v), then it is clear by \cite[Prop.~2.1]{es} (or \cite[Thm.~2.3]{b2}) in case (i), by \cite[Lemma 4.1]{lo} in cases (ii) and (iii), by Remark \ref{spi} in case (iv) and by Theorem \ref{bei} in case (v), that $\E$ is a rank $r$ Ulrich vector bundle for $(X,H)$ with the assigned values of $c_1(\E)^s, s \ge 1$. Moreover $\E$ is not big by Remark \ref{seg} in cases (i)-(iii), by Remark \ref{spi} in case (iv) and by Theorem \ref{bei} in case (v).

Vice versa let $\E$ be a rank $r$ Ulrich vector bundle for $(X,H)$ such that $\E$ is not big. Let $d = \deg X$.

If $d=1$ then $(X, \O_X(1)) = (\PP^3, \O_{\PP^3}(1))$ and it follows by \cite[Prop.~2.1]{es} (or \cite[Thm.~2.3]{b2}) that $\E = \O_{\PP^3}^{\oplus r}$, so that we are in case (i).

If $d=2$ then $(X, \O_X(1)) = (Q, \O_Q(1))$ whith $Q \subset \PP^4$ is a smooth quadric, and we are in case (iv) by Remark \ref{spi}. 

Assume from now on that $d \ge 3$.

If $c_1(\E)^2=0$, it follows by \cite[Cor.~1]{ls} that $(X, \O_X(1), \E)$ is as in (ii).

In the rest of the proof we will thus assume that $c_1(\E)^2 \ne 0$. 

By \cite[Thm.~1]{lo} we find that $X$ is covered by lines. By \cite[Thm.~1.4]{lp} it follows, using our assumption, that $(X, \O_X(1))$ is one of the following:
\begin{itemize}
\item [(1)] a linear $\PP^2$-bundle over a smooth curve;
\item [(2)] a linear $\PP^1$-bundle over a smooth surface;
\item [(3)] a Del Pezzo threefold;
\item [(4)] a quadric fibration over a curve.
\end{itemize}
In case (1), Theorem \ref{bei} gives that $(X, \O_X(1), \E)$ is as in (iii) when $c_1(\E)^3 = 0$ and as in (v) when $c_1(\E)^3>0$. 

Now we can assume that $(X, \O_X(1))$ is not as in (1). In particular Lemma \ref{riga} implies that $c_1(\E)^3=0$. 

In case (2), Proposition \ref{surf} gives that $(X, \O_X(1), \E)$ is as in (ii). 

We can thus assume that $(X, \O_X(1))$ is also not as in (2). 

To conclude we will now show that cases (3) and (4) do not occur. 

In case (3) only Del Pezzo threefolds of degree $d = 3, 4, 5$ remain, and these are excluded by \cite[Lemma 3.2]{lo} since they have $\rho(X)=1$.

In case (4) observe first that all fibers of the quadric fibration are irreducible. This was proved in \cite[\S 1]{io}. We give an alternative proof. Assume that there is a reducible fiber, say $M_1\cup M_2 \subset \PP^N$ where $M_1$ and $M_2$ are planes meeting along a line $R$. Then for any $x \in R$ we have that 
$$\langle M_1, M_2 \rangle \subseteq T_xX$$ 
hence they are equal. But this violates Zak's theorem on tangencies \cite[Thm.~0]{za}, \cite[Thm.~3.4.17]{laz1}.

Now, if all fibers of $\Phi_{\E} : X \to \mathbb G(r-1,rd-1)$ are $1$-dimensional, it follows by \cite[Prop.~3.2.1]{bs}, as in the proof of Proposition \ref{surf}, that $\Phi_{\E} : X \to \Phi_{\E}(X)$ is a linear $\PP^1$-bundle over a smooth surface, a contradiction. Therefore, by \cite[Thm.~1]{ls}, the only other possibility is that there is a fiber $F = \PP^2$ of $\Phi_{\E}$. Since $\PP^2$ does not surject onto a curve, it follows that $F$ is contained in a fiber of the quadric fibration. But then that fiber is reducible,  a contradiction.
\end{proof}

\end{document}